\documentclass[12pt]{article}
\usepackage{amsmath, amssymb,amsthm}
\usepackage{graphicx}

\usepackage{amsthm}

\usepackage{enumerate}

\newcommand{\ex}{\mathrm{ex}}
\newcommand{\Ex}{\mathrm{Ex}}
\newcommand{\up}{\mathrm{up}}

\newtheorem{theorem}{Theorem}
\newtheorem{lemma}{Lemma}

\newtheorem{prop}{Proposition}

\newtheorem{defin}{Definition}
\newtheorem{fact}{Fact}

\newtheorem{cor}{Corollary}

\def\le{\leqslant}
\def\ge{\geqslant}
\usepackage{lineno}

\begin{document}

\title{A hierarchy of maximal intersecting triple systems}

\author{
    Joanna Polcyn
    \\A. Mickiewicz University\\Pozna\'n, Poland\\{\tt joaska@amu.edu.pl}
  \and Andrzej Ruci\'nski\thanks{Research supported by the Polish NSC
grant 2014/15/B/ST1/01688. }
\\A. Mickiewicz University\\Pozna\'n, Poland\\{\tt rucinski@amu.edu.pl}  }

\date{\today}

\maketitle

\begin{abstract}
We reach beyond the celebrated theorems of Erd\H os-Ko-Rado and  Hilton-Milner, and, a recent
theorem of Han-Kohayakawa, and determine all  maximal intersecting triples systems. It turns out
that for each $n\ge7$ there are  exactly 15 pairwise non-isomorphic such systems (and 13 for
$n=6$). We present our result in terms of a hierarchy of Tur\'an numbers $\ex^{(s)}(n; M_2^{3})$,
$s\ge1$, where $M_2^{3}$ is a pair of disjoint triples. Moreover, owing to our unified approach, we provide
short proofs of the above mentioned results (for triple systems only).

The triangle $C_3$ is defined as $C_3=\{\{x_1,y_3,x_2\},\{x_1,y_2,x_3\}, \{x_2,y_1,x_3\}\}$.
Along the way we show that the largest intersecting triple system $H$ on $n\ge6$ vertices, which
is not a star and is triangle-free,  consists of $\max\{10,n\}$ triples.
 This facilitates our main proof's philosophy    which is to assume that $H$
 contains a copy of the triangle and analyze how the
remaining edges of $H$ intersect that copy.
\end{abstract}

\section{Introduction}\label{intro}

\emph{A hypergraph} is a synonym for \emph{set system} and in this context the sets are called \emph{edges}. The elements of all the sets are called \emph{vertices}.
We often identify the edge set of a hypergraph $H$ with the hypergraph itself but never forget about the underlying \emph{vertex set} $V(H)$.
 A hypergraph is called \emph{intersecting} if every two of its edges intersect. A hypergraph is \emph{$k$-uniform}, a $k$-graph, for short, if every edge has size $k$.

Although in this paper we prove results about triple systems, or $3$-uniform hypergraphs, we begin
with some definitions and results valid for all $k$-graphs,
$k\ge2$.

The celebrated Erd\"os-Ko-Rado theorem \cite{EKR} determines the maximum size of a $k$-uniform
intersecting family. Since we formulate this result in terms of the Tur\'an numbers, we need some
more definitions and notation.
  Given a family of $k$-graphs $\mathcal G$, we call a $k$-graph $H$ \emph{$\mathcal G$-free}
if for all $G \in \mathcal G$ we have  $G \nsubseteq H$.
\begin{defin}
    \rm For a family of $k$-graphs $\mathcal G$ and an integer $n\ge1$, the \textit{ Tur\'an number (of the 1st order)} is defined as
    $$
    \mathrm{ex}^{(1)}_k(n; \mathcal G):=\mathrm{ex}_k(n;\mathcal G)=\max\{|E(H)|:|V(H)|=n\;\mbox{ and $H$ is
        $\mathcal G$-free}\}.
    $$
    Every $n$-vertex $\mathcal G$-free $k$-graph with $\ex_k(n;\mathcal G)$
    edges is called  \emph{extremal (1-extremal) for~$\mathcal G$}.
    We denote by $\mathrm{Ex}_k(n;\mathcal G)=\mathrm{Ex}^{(1)}_k(n;\mathcal G)$
    the family of all $n$-vertex $k$-graphs which are extremal for $\mathcal G$.
\end{defin}

In \cite{JPRr} the authors introduce a hierarchy of Tur\'an numbers, where in each  generation we consider only $k$-graphs
which are not sub-$k$-graphs of extremal $k$-graphs from all previous generations. The next
definition is iterative.

\begin{defin}\rm  For a family of $k$-graphs $\mathcal G$ and  integers $s,n\ge1$,
    \textit{the Tur\'an number of the $(s+1)$-st order} is defined as
    \begin{eqnarray*}
        \mathrm{ex}^{(s+1)}_k(n;\mathcal G)=\max\{|E(H)|:|V(H)|=n,\; \mbox{$H$ is
            $\mathcal G$-free, and }\\
        \forall H'\in \mathrm{Ex}^{(1)}_k(n;\mathcal G)\cup...\cup\mathrm{Ex}^{(s)}_k(n;\mathcal G),  H\nsubseteq H'\},
    \end{eqnarray*}
    if such a $k$-graph $H$ exists.
    An $n$-vertex $\mathcal G$-free $k$-graph $H$ is called \textit{(s+1)-extremal for} $\mathcal G$ if $|E(H)| = \mathrm{ex}^{(s+1)}_k(n;\mathcal G)$ and  $\forall H'\in \mathrm{Ex}^{(1)}_k(n;\mathcal G)\cup...\cup\mathrm{Ex}^{(s)}_k(n;\mathcal G),  H\nsubseteq H'$; we denote by $\mathrm{Ex}^{(s+1)}_k(n;\mathcal G)$ the family of $n$-vertex $k$-graphs which are $(s+1)$-extremal for $\mathcal G$.
\end{defin}

\noindent We will often write $\mathrm{ex}^{(s)}_k(n;G)$ for $\mathrm{ex}^{(s)}_k(n;\{G\})$ and
$\mathrm{Ex}^{(s)}_k(n;G)$ for $\mathrm{Ex}^{(s)}_k(n;\{G\})$.

A {\it star} is a hypergraph with a vertex, called its center, contained in all the edges.
Obviously, a star is intersecting. An $n$-vertex, $k$-uniform star with ${n-1\choose k-1}$ edges is
called {\it full} and denoted by $S^k_n$. Let $M_2^k$ be a $k$-graph consisting of two disjoint
edges.
    \begin{theorem}[\cite{EKR}]\label{erdos}
        For $n\ge 2k$,
        $$\mathrm{ex}_k(n;M^k_2)=\binom{n-1}{k-1}.$$
         Moreover, for $n\ge2k+1$, $\mathrm{Ex}_k(n;M^k_2)=\{S^k_n\}$.
    \end{theorem}
A historically first example of a Tur\'an number of the 2nd order is due to Hilton and Milner
\cite{HM} (see \cite{FF2} for a simple proof). They determined the maximum size of an intersecting
$k$-graph which is not a star.

\begin{theorem}[\cite{HM}]\label{nti}
    For $n\ge 2k+1$,
    $$\mathrm{ex}^{(2)}_k(n;M^k_2)={n-1\choose k-1}-{n-k-1\choose k-1}+1.$$
     Moreover, for $k=3$, $\mathrm{Ex}^{(2)}_3(n;M^3_2)=\{H_1(n),H_2(n)\}$,
    while for $k\ge 4$, $\mathrm{Ex}^{(2)}_k(n;M^k_2)=\{H^k_1(n)\}$.
\end{theorem}

The two 2-extremal $3$-graphs  $H_1(n)$ and $H_2(n)$ appearing in Theorem \ref{nti} are defined
later in this section. For the definitions of $H^k_1(n)$ for arbitrary $k\ge4$ see \cite{HM}.

Recently, the third order Tur\'an number for $M^k_2$ has been established for arbitrary $k$ by Han
and Kohayakawa in \cite{HK}.
\begin{theorem}[\cite{HK}]\label{ntntif}
    For $k\ge 3$ and $n\ge 2k+1$,  we have
    $$\mathrm{ex}_k^{(3)}(n;M^k_2)={n-1\choose k-1}-{n-k-1\choose k-1}-{n-k-2\choose k-2}+2.$$
\end{theorem}
Han and Kohayakawa have also determined the 3-extremal $k$-graphs which are not shown here. Below
we define $H_3(n)$, the only 3-extremal 3-graph for $M_2^3$ and $n\ge7$.

A natural question arises if this process terminates. In other words, is the number of maximal
intersecting $k$-graphs  finite, that is, independent of $n$, the number of vertices? This question
has been answered positively already in \cite{PLWomen} (see also \cite{Tuza}) but no extremal
hypergraphs were given.

In this paper we produce explicitly the entire spectrum of maximal,
intersecting 3-graphs and arrange them by means of the ordered Tur\'an numbers for the matching
$M_2^3$ (Proposition \ref{6} and Theorem \ref{main} below). In particular, we find all, pairwise non-isomorphic, maximal, intersecting 3-graphs on six vertices (Proposition \ref{6}). Although, for $n\ge7$, the Tur\'an numbers  of the
first, second, and third order are already known (and are stated above), for the sake of
unification, we include them into our main result. In addition, we determine the complete  Tur\'an
hierarchy for non-intersecting 3-graphs which are triangle-free (Corollary \ref{MC} in Section
\ref{CM}).

Before stating our results, we need to define several specific 3-graphs which  turn out to be
extremal. In our description, we will put an emphasis on the vertex covers.
A subset of vertices $T$ of a hypergraph $H$ is called \emph{a vertex cover} if it has nonempty intersection with every edge of $H$. We denote by $\tau(H)$ the size of the smallest vertex cover of $H$.
 Clearly, every edge of an intersecting hypergraph is its vertex cover, so, for an intersecting $k$-graph $H$ we have $1\le\tau(H)\le k$.

 For a  subset
$A\subset V$, set $\up(A)=\left\{f\in {V\choose k}: f\supset A\right\}$.
Observe that if $T$ is a vertex cover of a \emph{maximal} intersecting $k$-graph $H$, then $\up(T)\subseteq H$.

\medskip

 Let $x,y,z,v,w,u\in V$ be six different vertices of $V$, $|V|=n$. We define
$$
H_1(n)=\up(\{x,y\})\cup \up(\{x,z\})\cup\up(\{x,v\})\cup\left\{\{y,z,v\}\right\}
$$
and
$$
H_2(n)=\up(\{x,y\})\cup \up(\{x,z\})\cup\up(\{y,z\}).
$$
Note that for  $i=1,2$, $M_2\not\subset H_i(n)$ and $|H_i(n)|=3n-8$. Next, let
$$
H_3(n)=\up(\{x,y\})\cup \up(\{x,z\})\cup\{\{x,v,w\},\{y,z,w\},\{y,z,v\}\}.
$$
Note that $M_2\not\subset H_3(n)$ and $|H_3(n)|=2n-2$. Further, let
$$
H_4(n)=\up(\{x,y\})\cup \{\{x,v,z\},\{x,w,z\},\{x,v,w\},\{y,z,w\},\{y,z,v\},\{y,v,w\}\},
$$
$$
H_5(n)=\up(\{x,y\})\cup \{\{x,v,z\},\{x,w,u\},\{x,v,w\},\{y,z,w\},\{y,u,v\},\{y,v,w\}\},
$$
and
$$
H_6(n)=\up(\{x,y\})\cup \{\{x,v,z\},\{x,w,u\},\{x,v,w\},\{y,z,w\},\{y,u,v\},\{x,z,u\}\}.
$$
Note that for  $i=4,5,6$, $M_2\not\subset H_i(n)$ and $|H_i(n)|=n+4$.

Observe also that $\tau(H_i(n))=2$ for $i=1,\dots,6$. The minimal vertex covers can be easily identified, as they are exactly the 2-element sets which are the arguments of the operator $\up(\cdot)$ appearing in all 6 definitions above. For instance, $H_1(n)$ has three minimal vertex covers, $\{x,y\}, \{x,z\},\{x,v\}$, while $H_6(n)$ has just one, $\{x,y\}$.

\medskip

Next, we define five more intersecting 3-graphs, $H_i(n)$, $i=7,\dots,11$, all with just 10 edges and spanned on 6 vertices. As now $\tau(H_i(n))=3$, the remaining $n-6$ vertices are isolated.
Below we use notation $H\cup sK_1$ to
designate the 3-graph obtained from a 3-graph $H$ by adding $s$ isolated vertices.
We find it convenient and in line with the forthcoming proof to base their description on
the notion of a triangle whose copy they all contain. Let $ U=\{x_1,x_2,x_3,y_1,y_2,y_3\} $. We call the cycle
$$
C_3=\{\{x_i,y_j,x_k\}:\{i,j,k\}=\{1,2,3\}\}
$$
\emph{a triangle}. Further, let
$$ A_1=\{\{x_i,y_i,y_j\}:\{i,j\}\subset \{1,2,3\}\}, $$
$$
A_2=\{\{x_i,x_j,y_j\}:\{i,j\}\subset \{1,2,3\}\}, $$
$$ A_3=(A_1\setminus \{x_1,y_1,y_2\})\cup
\{x_2,x_3,y_3\}, $$ and
$$
A_4=(A_1\setminus \{\{x_1,y_1,y_2\},\{x_1,y_1,y_3\}\})\cup\{ \{x_2,x_3,y_3\},\{x_2,x_3,y_2\}\}.
$$
Note, that $A_4=(A_3\setminus \{x_1,y_1,y_3\})\cup \{x_2,x_3,y_2\}$.

\bigskip

We define the following 3-graphs on $U$:
$$ H_7(6)=C_3\cup \{\{x_1,x_2,x_3\}\}\cup A_1, $$ and,
for $i=8,9,10,11$,
$$ H_i(6)=C_3\cup\{ \{y_1,y_2,y_3\}\}\cup A_{i-7}. $$
Finally, for $i=7,\dots,11.$ and $n\ge 7$, set
$$H_i(n)=H_i(6)\cup(n-6)K_1.$$

Observe that the  3-graphs $S_n$ and $H_i(n)$, $i=1,\dots,11$, are all \emph{maximal} with respect to
being intersecting, that is, adding a new edge always results in the appearance of a copy of $M_2$.
Consequently, they are mutually not sub-3-graphs of each other.

\bigskip

For $k=3$ we  suppress the superscript $^3$ from the notation of $3$-graphs, i.e., $M_2^3=M_2$ and
$S_n^3=S_n$. We also suppress the subscript $_3$ in all Tur\'an related notation. It is already
known (see Theorem \ref{erdos}) that for each $n\ge6$,  the  full  star $S_n$  is a
$1\textrm{-extremal}$ intersecting 3-graph for $M_2$. We are now ready to identify all Tur\'an
numbers $\mathrm{ex}^{(s)}(n;M_2)$ together with the sets of $s$-extremal 3-graphs
$\Ex^{(s)}(n;M_2)$, $s\ge1$. Let us fix the vertex set $V$, $|V|=n$. For $n\le5$ every 3-graph is
intersecting and thus $\ex^{(1)}(n;M_2)=\binom n3$, the only 1-extremal 3-graph is the clique
$K_n$, and the higher order Tur\'an numbers $\ex^{(s)}(n;M_2)$, $s\ge2$, do not exist.

If $n=6$, each triple in $\binom V3$ intersects all
 other triples except its complement. Therefore, we may arrange all 20 triples into 10 pairs (an edge
 and
 its complement) and from each such a pair choose arbitrarily one triple to get a maximal intersecting $3$-graph,
consisting of 10 edges. This yields $2^{10}$ 3-graphs, among which we found 13 pairwise
non-isomorphic ones, as specified in Proposition \ref{6} below.

\begin{prop}\label{6}
 We have $\ex^{(1)}(6;M_2)=10$ and
 $$\Ex^{(1)}(6;M_2)=\{S_6, K_5\cup K_1,\;H_i(6),\;i=1,\dots,11\},$$
 where the $H_i(6)$'s are defined above.
\end{prop}
As  every intersecting 3-graph on 6 vertices is a sub-3-graph of one of the above 13 extremal
3-graphs, there are no higher order Tur\'an numbers for $n=6$.

Things change dramatically for $n\ge7$. First notice that maximal intersecting $3\textrm{-graphs}$ on $n\ge7$ vertices can be
obtained  from any of the 13 6-vertex  3-graphs appearing in Proposition \ref{6} by adding  all triples containing any of
their vertex covers. This way we obtain  3-graphs $S_n$, $K_5\cup (n-5)K_1$, and $H_i(n)$, $i=1,\dots,11.$ As it turns out there are only two
other maximal intersecting 3-graphs for $n\ge7$.

Let $F_7$ be the Fano plane, that is a 3-graph on 7 vertices obtained from the triangle $C_3$ by adding
one new vertex $z$ and four new edges: $\{x_i,z,y_i\}$, $i=1,2,3$, and $\{y_1,y_2,y_3\}$. Further,
let $F_{10}$ be a 3-graph obtained from the triangle $C_3$ by adding one more vertex $z$ and 7 new edges:
$\{x_1,x_2,x_3\},\{x_1,x_2,z\},\{x_1,z,x_3\},\{z,x_2,x_3\}$, and $\{x_i,y_i,z\}$, $i=1,2,3$.


\begin{theorem}\label{main}
    For  $n\ge 7$,

    \begin{enumerate}
\item $\mathrm{ex}^{(1)}(n;M_2)={n-1\choose 2}$  and  $\Ex^{(1)}(n;M_2)=\{S_n\}$,

\item $\mathrm{ex}^{(2)}(n;M_2)=3n-8$  and  $\Ex^{(2)}(n;M_2)=\{H_1(n),H_2(n)\}$,

\item $\mathrm{ex}^{(3)}(n;M_2)=2n-2$  and   $\Ex^{(3)}(n;M_2)=\{H_3(n)\}$

\item $\mathrm{ex}^{(4)}(n;M_2)=n+4$  and  $\Ex^{(4)}(n;M_2)=\{H_4(n), H_5(n), H_6(n)\}$

\item  $\mathrm{ex}^{(5)}(n;M_2)=10$  and  $\Ex^{(5)}(n;M_2)=\{K_5\cup (n-5)K_1, F_{10}\cup (n-7)K_1,\\
H_i(n): i=7,\dots, 11\}$,

\item $\mathrm{ex}^{(6)}(n;M_2)=7$  and   $\Ex^{(6)}(n;M_2)=\{F_7\cup (n-7)K_1\}$.
 \end{enumerate}
The Tur\'an numbers $\mathrm{ex}^{(s)}(n;M_2)$ do not exist for $s\ge7$.

\end{theorem}


\section{Proofs }\label{proof}

In this section we present all our proofs. We begin with some general simple observations about the structure of maximal intersecting hypergraphs. In Subsection \ref{CM} we determine all Tur\'an numbers $\ex^{(s)}(n;\{M_2,C_3\})$, $n\ge6$, $s\ge1$, and accompanying them $s$-extremal 3-graphs for the pair $\{M_2,C_3\}$ (see Corollary \ref{MC}). The remaining two subsections contain the proofs of Proposition \ref{6} and Theorem \ref{main}, respectively.

\subsection{The structure of maximal intersecting hypergraphs}\label{struc}

 Recall that a subset of vertices $T$ of a hypergraph $H$ is called \emph{a vertex cover} if it has nonempty intersection with every edge of $H$ and that $\tau(H)$ stands for the size of the smallest vertex cover of $H$.
For $k\ge2$ and $n\ge2k$, let $H$ be an $n$-vertex, maximal intersecting $k$-graph. Clearly, every edge of $H$ is its vertex cover, so $1\le\tau(H)\le k$. We have already mentioned that if $T$ is a vertex cover of $H$ then, by maximality, $\up(T)\subseteq H$. As an immediate consequence, we deduce the following useful observation.

\begin{fact}\label{intcov}
For $k\ge2$ and $n\ge2k$, let $H$ be an $n$-vertex, maximal intersecting $k$-graph. Then the family of all vertex covers of $H$ is intersecting itself.
\end{fact}
\proof Suppose $T_1$ and $T_2$ are two disjoint vertex covers of $H$. Then, since $n\ge2k$, there are $e_i\in\up(T_i)\subseteq H$, $i=1,2$, such that $e_i\cap e_2=\emptyset$, a contradiction. \qed

\medskip

Our next observation will be of great help in  the proof of the main theorem in Subsection \ref{pm}.
We call a subset $U\subseteq V(H)$ \emph{a heart of} $H$ if every two edges of $H$ intersect on $U$, that is, if for all $e,f\in H$, we have $e\cap f\cap U\neq\emptyset$. \emph{The induced} sub-$k$-graph $H[U]$ consists of all edges of $H$ which are contained in $U$, that is, $H[U]=\{e\in H:\; e\subset U\}.$
Trivially, for every $U\subseteq V(H)$, $H[U]$ is intersecting as well. It turns out that every reasonably large heart of $H$  is also maximal.

\begin{fact}\label{maxheart}
For $k\ge2$ and $n\ge2k$, let $H$ be an $n$-vertex, maximal intersecting $k$-graph. If $U$ is a heart of $H$, $|U|\ge2k$, then $H[U]$ is maximal intersecting $k$-graph itself.
\end{fact}
\proof Suppose not. Then there exists a $k$-element set $T\subset U$ and an edge $e\in H$ such that  $T\cap e=\emptyset$ but  $H[U]\cup\{T\}$ is still intersecting, i.e. $T$ is a vertex cover of $H[U]$. Now, $e\setminus U$ is a vertex-cover of $H$, and so, $\up(e\setminus U)\subseteq H$. In particular, since $|U|\ge2k$, there is an edge $f\in\up(e\setminus U)$ such that $f\subset U$, that is, $\in H[U]$, and $f\cap T=\emptyset$. This is, however, a contradiction with the assumption that $T$ is a vertex cover of $H[U]$. \qed

\subsection{Triangle-free intersecting 3-graphs}\label{CM}

Recall that a triangle $C_3$ consists of a vertex set $U=\{x_1,x_2,x_3, y_1,y_2,y_3\}$ and the edge
set
$$
C_3=\{\{x_i,y_j,x_k\}: \{i,j,k\}=\{1,2,3\}\}.
$$
Thus, the vertices $x_1, x_2, x_3$ are of degree two in $C_3$, while $y_1,y_2,y_3$ are of degree
one.

The Tur\'an numbers for $C_3$ were determined in \cite{FF} for $n\ge75$ and in \cite{CK} for all
$n$.

\begin{theorem}[\cite{CK}]\label{kahn}
For $n\ge 6$, $\ex^{(1)}(n;C)={n-1\choose 2}$. Moreover, for $n\ge 8$, $\Ex^{(1)}(n;C_3)=\{S_n\}$,
for $n=7$, $\Ex^{(1)}(7;C_3)=\{S_7, \up(\{u,v\})\cup \binom{V\setminus\{u,v\}}3\}$, and for $n=6$,
$\Ex^{(1)}(6;C_3)=\{S_6, K_5\cup K_1\}$. \qed
\end{theorem}

Define  $H_0(n)$ as a 3-graph obtained from a copy of $K_4$ on the set of vertices $\{x,y,z,v\}$,
by adding to it all the edges of the form $\{x,y,w\}$ where $w\notin \{x,y,z,v\}$, namely,
$$
H_0(n)=\up(\{x,y\})\cup\{\{x,z,v\},\{y,z,v\}\}.
$$
Note that $|H_0(n)|=n$, $H_0(n)\subset G_i(n)$ for $i=1,\dots,5$, and $H_0(n)$ is
$\{M_2,C_3\}$-free. The next lemma plays an important role in the proof of Theorem \ref{main}.

\begin{lemma}\label{mc}
    For $n\ge 6$,  if $H$ is an $n$-vertex $\{M_2,C_3\}$-free 3-graph not contained in the star $S_n$,
then $H\subseteq K_5\cup (n-5)K_1$ or $H\subset H_0(n)$.
\end{lemma}

\begin{proof}
Let $H$ be a $\{M_2,C_3\}$-free
 3-graph $H$ on the set of vertices $V$, $|V|=n$, which is not a star.
 We will show that  $H$ is a sub-3-graph of either $K_5\cup (n-5)K_1$ or $H_0(n)$.

Let $P_2$ denote a 3-graph consisting of two edges sharing exactly one vertex.
  We may assume that $P_2\subset H$, because otherwise, every two edges of $H$ would intersect in exactly
   two vertices and, consequently, $H\subseteq \up(\{x,y\})$, for some two vertices
   $x,y\in V$, or $H\subseteq K_4\cup (n-4)K_1$, implying that
    $H\subseteq H_0(n)$.
   Let us set
 $P_2=\{e_1,e_2\}$, $e_1\cap e_2=\{x\}$, $U=V(P_2)$, and $W=V\setminus U$, $|W|=n-5$.
 If all the edges of $H$ are contained in $U$,
      then $H\subseteq K_5\cup (n-5)K_1$. Therefore, in the rest of the proof we
       will be assuming that there exists an edge $f\in H$ with $f\cap W\neq \emptyset$.

        As $H$ is intersecting, every edge $f\in H$ must, in particular,
         intersect $e_1$ and $e_2$. Therefore, since $C_3\nsubseteq H$, every edge $f\in H$
  with $f\cap W\neq \emptyset$ contains vertex $x$. But $H\nsubseteq S_n$ and hence
  there exists an edge $h\in H$ such that $x\notin h$. As explained above, $h\subset U$.
   Without loss of generality we may assume that $|h\cap e_i|=i$ for $i=1,2$ and let
    $h\cap e_1=\{y\}$, $h\cap e_2=\{z,v\}$. Then the edges $h$ and $e_1$ form another copy of $P_2$
     and using the same argument as above, every edge $f\in H$ with $f\cap W\neq\emptyset$
      must contain vertex $y$. Consequently, all the edges of $H$ satisfying $f\cap W\neq\emptyset$
       are of the form $\{x,y,w\}$, where $w$ is an arbitrary vertex of $W$. One can check that
         adding to $H$ any triple $e\in\binom U3$, except for $\{x,y,z\}$ and $\{x,y,v\}$, creates, together with an edge $\{x,y,w\}$, $w\in W$, either a
          triangle or a pair of disjoint edges. Hence, $H\subseteq H_0(n)$. \end{proof}

An immediate corollary of Theorem \ref{kahn} and Lemma \ref{mc}  gives the Tur\'an numbers for the
pair $\{M_2,C_3\}$. Note that the 3-graph $\up(\{u,v\})\cup \binom{V\setminus\{u,v\}}3$ on 7
vertices contains $M_2$ and therefore is `disqualified' here.

\newpage

\begin{cor}\label{MC} The complete Tur\'an hierarchy for the pair $\{M_2,C_3\}$ is as follows:
\begin{enumerate}
\item For $n\le 5$,
$$\ex^{(1)}(n;\{M_2,C_3\})=\binom n3,\qquad \Ex^{(1)}(n;\{M_2,C_3\})=\{K_n\},$$
 and
 $\ex^{(s)}(n;\{M_2,C_3\})$ does not exist for $s\ge2$.
\item For $n=6$,
$$\ex^{(1)}(6;\{M_2,C_3\})=10,\qquad \Ex^{(1)}(6;\{M_2,C_3\})=\{S_6,K_5\cup K_1\},$$
 $$\ex^{(2)}(6;\{M_2,C_3\})=6, \qquad\Ex^{(2)}(6;\{M_2,C_3\})=\{H_0(6)\},$$ and
 $\ex^{(s)}(6;\{M_2,C_3\})$ does not exist for $s\ge3$.
\item For $n\ge7$,
$ \ex^{(1)}(n;\{M_2,C_3\})={n-1\choose 2}, \qquad \Ex^{(1)}(n;\{M_2,C_3\})=\{S_n\}.
    $
    \item
 For $n=10$, $\Ex^{(2)}(10;\{M_2,C_3\})=\{K_5\cup 5K_1, H_0(10)\}$ and
    $\ex^{(s)}(10;\{M_2,C_3\})$ does not exist for $s\ge3$.
    \item For $n\ge7$, $n\neq10$,
    $$
    \ex^{(2)}(n;\{M_2,C_3\})=\max\{10,n\}, \qquad \ex^{(3)}(n;\{M_2,C_3\})=\min\{10,n\},$$
     $$ \Ex^{(2)}(n;\{M_2,C_3\})\cup \Ex^{(3)}(n;\{M_2,C_3\})=\{K_5\cup (n-5)K_1, H_0(n)\},
    $$
    and $\ex^{(s)}(n;\{M_2,C_3\})$ does not exist for $s\ge4$.

    \end{enumerate}
\end{cor}


\subsection{Proof of Proposition \ref{6}} To prove Proposition \ref{6}, we need to show that among all $1024$ labeled intersecting
$3\textrm{-graphs}$ on 6 vertices there are exactly 13 isomorphism types listed therein. We already know from
Theorem \ref{kahn}, that only two of these 3-graphs are $C_3$-free, namely, $S_6$ and $K_5\cup K_1$.

 Not without a reason,
we classify the remaining 3-graphs $H$, that is, those containing $C_3$, with respect to the number
of  vertex covers of size 2. Since every edge in an intersecting 3-graph is its vertex cover, the
minimum size of a vertex cover is either 2 or 3 (1 is  impossible due to the presence of $C_3$).
Let us call a cover set of size 2, simply a \emph{2-cover}.

 Let $C$ be a copy of $C_3$ on vertex set $U=\{x_1,x_2,x_3, y_1,y_2,y_3\}$ and with edge set
$C=\{\{x_i,y_j,x_k\}: \{i,j,k\}=\{1,2,3\}\}. $ Note that there are six different 2-covers of
$C$: $T_i=\{x_i,y_i\}$, $i=1,2,3$, $T_4=\{x_1,x_2\}$, $T_5=\{x_2,x_3\}$ and $T_6=\{x_1,x_3\}$.
Thus, the 2-covers of $H$  must be among these  six. But, by Fact \ref{intcov},
there are no disjoint 2-covers in $H$, so there are at most three 2-covers in $H$. Recall also that if a pair $T$ is a
2-cover in a maximal intersecting 3-graph $H$, then   $H\supseteq\up(T)$.

\noindent{\bf Case 1:} there are three 2-covers in $H$. Up to isomorphism, there are only two
possibilities: either $T_4,T_1,T_6$ are the 2-covers in $H$ or  $T_4,T_5,T_6$ are the 2-covers in $H$. In
each case, there are exactly ten triples belonging to $C$ or containing at least one of these
2-covers, so there are no more triples in $H$. Then, it is easy to check that $H$ is isomorphic to $H_1(6)$ (in the first
case) or $H$ is isomorphic to $H_2(6)$ (the second case).

\smallskip

\noindent{\bf Case 2:} there are exactly two 2-covers in $H$. Without loss of generality
either $T_1$ and $T_4$ or $T_4$ and $T_6$ are the unique 2-covers in $H$. In both cases $H$ must
contain all triples
 that contains at least one of these sets. Therefore, in the first case  $H$ contains the following triples:
$T_1\cup\{v\}$, where $v\in \{x_2,x_3,y_2,y_3\}$ and $T_4\cup \{v\}$ where $v\in \{x_3,y_2\}$.
Moreover, since $T_6$ is not a 2-cover of $H$, there must be in  $H$ a triple disjoint from $T_6$
but touching both, $T_1$ and $T_4$, and all the edges of $C$. There is only one such triple,
namely $\{x_2,y_1,y_2\}$. We have $|H|= 10$ and $H\cong H_3(6)$.

If $T_4,T_6$ are the unique 2-covers in $H$, then, as above,  $H$ contains all the triples
containing $T_4$ or $T_6$, namely $T_4\cup \{v\}$, where $v\in \{x_3,y_1, y_2\}$, and $T_6\cup
\{v\}$, where $v\in \{y_1, y_3\}$. Again, since neither $T_1$ nor $T_5$ is a 2-cover in $H$,  $H$
must contain two intersecting triples disjoint from $T_1$ and $T_5$, respectively, but touching
both 2-sets, $T_4$ and $T_6$, and all the edges of $C$. Up to isomorphism there is only one
possibility for this: $\{x_2,x_3,y_3\}\in H$ and $\{ x_1,y_1,y_3\}\in H$. As before we have
$|H|=10$ and $H\cong H_3(6)$.

\smallskip

\noindent{\bf Case 3:} there is exactly one 2-cover in $H$. We claim that $H$ is isomorphic to one of $H_4(6), H_5(6)$, and $H_6(6)$.
Note that all three have a similar structure: there is one 2-cover $\{x,y\}$, and so each consists of all four edges containing it, plus the edges adjacent to $x$ but not $y$ and vice versa. Let $L_i(x)$ be the set of pairs making an edge with $x$ but not with $y$ in $H_i(6)$, $i=4,5,6$, and we define $L_i(y)$ analogously. Referring to the definitions of $H_i(6)$, $i=4,5,6$, in Introduction, we see that $L_4(x)=L_4(y)$ is the (graph) triangle on $z,v,w$; $L_5(x)$ is the path $zvwu$, while $L_5(y)$ is the path $zwvu$, so these two paths share the middle pair; finally, $L_6(x)$ is the (graph) 4-cycle $zvwuz$, while $L_6(y)$ is a (graph) matching consisting of the two diagonals of that 4-cycle, $zw$ and $uv$.

Up to isomorphism there are only two subcases. Either $T_1$ or $T_4$ is the unique 2-cover of $H$.
Assume first it is $T_1$. Then we do not need to worry about $T_2,T_3$, and $T_5$ as they are all disjoint from $T_1$ (recall that $\up(T_1)\subset H$). To prevent $T_4$ and $T_6$ from being also 2-covers of $H$, there must be an edge or edges in $H$ disjoint from those two pairs, but intersecting $T_1$.

Assume first that $\{y_1,y_2,y_3\}\in H$ is such an edge. It takes care of both, $T_4$ and $T_6$. As $H$ is maximal, it must also contain two more edges, say $\{x_1,x_3,y_3\}$ and $\{x_1,x_2,y_2\}$ (there are 3 more options here, in which either of these two edges is replaced by its complement; we leave their analysis to the reader). So, there are 10 edges altogether. We see that the pairs making an edge with $x_1$, but not with $y_1$, form the 4-cycle $x_2y_2x_3y_3x_2$, and there are only two pairs, $x_2x_3$ and $y_2y_3$ making an edge with $y_1$, but not with $x_1$. Thus, $H\cong H_6(6)$ (with  $x:=x_1$ and $y:=y_1$).

If $\{y_1,y_2,y_3\}\notin H$, then its complement $\{x_1,x_2,x_3\}\in H$. The only edges which `exclude' $T_4$ and $T_6$ are $\{x_3,y_1,y_3\}$ and $\{x_2,y_1,y_2\}$, respectively. So, again, we have 10 edges in $H$, but this time the pairs making an edge with $x_1$, but not with $y_1$, form the path $y_3x_2x_3y_2$, while in the opposite case, it is the path $y_2x_2x_3y_3$. Thus, $H\cong H_5(6)$.

Assume now that the unique 2-cover of $H$ is $T_4$. We need to `exclude' four other 2-covers of $C$, namely, $T_1,T_2,T_5$, and $T_6$, from being present in $H$. There are four subcases. In the first one, let $\{x_1,y_1,y_3\}\in H$ and $\{x_2,y_2,y_3\}\in H$ (the remaining 3 cases come from negating one or both clauses in this conjunction).
The first of these two edges excludes $T_2$ and $T_5$, while the other does the same to $T_1$ and $T_6$. By maximality there are  two more edges in $H$, say $\{x_1,y_1,y_2\}$ and $\{x_2,y_2,y_1\}$ (again, we skip 3 more cases with involving the complements). A similar analysis of the graph links of $x:=x_1$ and $y:=x_2$ leads to a conclusion that, again, $H\cong H_5(6)$.

Consider now the subcase when $\{x_1,y_1,y_3\}\notin H$ and $\{x_2,y_2,y_3\}\notin H$. Then, the complements $\{x_2,x_3,y_2\}\in H$ and $\{x_1,x_3,y_1\}\in H$ exclude $T_1$ and $T_2$, respectively. In addition, we must also have $\{x_1,y_1,y_2\}\in H$ and $\{x_2,y_1,y_2\}\in H$ which take care of $T_5$ and $T_6$, respectively. A similar analysis reveals that  $H\cong H_4(6)$.

The remaining two subcases are symmetrical, so we consider only one of them. Let $\{x_2,x_3,y_2\}\in H$ and $\{x_2,y_2,y_3\}\in H$. These two edges, together with $\{x_1,y_1,y_2\}$ and $\{x_1,x_3,y_3\}$, exclude all four forbidden 2-covers, $T_1$, $T_6$, $T_5$, and $T_2$. A quick look at the links of $x_1$ and $x_2$ shows that this time  $H\cong H_5(6)$.

\smallskip

\noindent{\bf Case 4:} there is no 2-cover  in $H$. This means that for each 2-cover of $C$,
$T_i$, $i=1,\dots,6$, there is an edge in $H$ disjoint from it. A tedious case by case analysis of
the $2^7$ remaining choices between triples and their complements (we have already made three
choices by implanting the triangle $C$ in $H$) leads always to one of the 3-graphs $H_i(6)$,
$i=7,\dots,11$. We omit the details.

\subsection{Proof of Theorem \ref{main}}\label{pm}

    Let $H$ be a \emph{maximal} $M_2$-free 3-graph with $V(H)=V$ and $|V|=n\ge 7$,
 not contained in a star $S_n$ and $K_5\cup (n-5)K_1$. Then by Lemma \ref{mc} we have
 $C_3\subset H$ (note that since $H_0(n)$ is not maximal, $H\nsubseteq H_0(n)$). We say that a copy $C$ of the triangle, $C\subset H$, is \emph{a triangular heart} of $H$ if $V(C)$ is a heart of $H$ (see the definition of heart in Subsection \ref{struc}). Let $C\subset H$ be a copy of $C_3$ in $H$. Set
    $
    U=V(C)=\{x_1,x_2,x_3,y_1,y_2,y_3\}\subset V,
    $
and, $
C=\{\{x_i,y_j,x_k\}: \{i,j,k\}=\{1,2,3\}\}. $ Further, let $ W=V\setminus U$, $|W|=n-6. $ Since
$H$ is intersecting, every edge of $H$ intersects $U$ on at least 2 vertices.

\noindent{\bf Case 1:}  $H$ has no triangular heart. Then there exist two edges $h_1,h_2\in H$ with $h_1\cap
h_2\cap U=\emptyset$. Without loss of generality let $h_1=\{x_1,y_1,w\}$, where $w\in W$. We start
with the case $h_2=\{x_2,y_2,w\}$ (the case $h_2=\{x_3,y_3,w\}$ is symmetrical). There exists only
one 2-cover of the edge set $C\cup \{h_1,h_2\}$, namely $T=\{x_1,x_2\}$. Therefore, all the edges
$h\in H$ such that $h\cap (W\setminus \{w\})\neq \emptyset$ contain $T$. There are only two triples
which are disjoint from $T$ and  intersect all the edges of $C\cup \{h_1,h_2\}$, namely
$h_3=\{x_3,y_3,w\}$ and $h_4=\{y_1,y_2,y_3\}$.

If for $i=3,4$, $h_i\notin H$, then the triangle
$C'=\{\{x_1,w,y_1\},\{y_1,x_2,x_3\},\{x_3,y_2,x_1\}\}$ is a triangular heart of $H$, a contradiction.
 Therefore at least one of the edges, $h_3$ or $h_4$ belongs
to $H$.
    If both $h_3\in H$ and $h_4\in H$, then $H[U\cup\{w\}]$ is the Fano plane $F_7$ which is a maximal
      intersecting family. Moreover, there are no 2-covers in $F_7$, so there are no other edges in
      $H$, and we conclude that $H=F_7\cup (n-7)K_1$.

 Next, let $h_3\in H$ and $h_4\notin H$.
 Then, since $H$ is maximal, it contains four more edges, $\{x_1,x_2,x_3\}$, $\{x_1,x_2,w\}$, $\{x_1,x_3,w\}$
 and $\{x_2,x_3,w\}$, and so, $H=F_{10}\cup (n-7)K_1$. Otherwise $h_4\in H$ and, since $h_3\notin H$,
 there are four more edges in $H$, $\{x_1,y_1,y_2\}$, $\{x_1,y_1,x_2\}$, $\{x_2,y_2,y_1\}$ and
 $\{x_2,y_2,x_1\}$. Again, $H=F_{10}\cup (n-7)K_1$.

    Now we move to the case when for $i=2,3$, $\{x_i,y_i,w\}\notin H$. Then, since $H$
  is an intersecting family not containing a triangular heart, we must have $h_2=\{x_2,x_3,w\}\in H$. This time there
  are two intersecting 2-covers of $C\cup \{h_1,h_2\}$, $T_1=\{x_1,x_2\}$ and $T_2=\{x_1,x_3\}$. Like above,
  since there is no triangular heart in $H$, there must be in $H$ two edges $h_3$ and $h_4$ such that
  $h_3\cap T_1=\emptyset$ and
     $h_4\cap T_2=\emptyset$. We have no other choice but set $h_3=\{x_3,y_3,y_1\}$ and
      $h_4=\{x_2,y_2,y_1\}$. So, there are only three more edges in $H$, $\{x_1, x_2, x_3\}$,
       $\{x_1, x_2, y_1\}$ and $\{x_1, x_3, y_1\}$.
        Thus,  $H=F_{10}\cup (n-7)K_1$ again.
        As $|F_7|=7$ and $|F_{10}|=10$, these two 3-graphs do not
        play any role in establishing the first four Tur\'an numbers for $M_2$.

    \bigskip

\noindent{\bf Case 2:} all the edges of $H$ intersect each other on $U=V(C)$, that is, $C$ is a triangular heart of~$H$.
By Fact \ref{maxheart}, the induced sub-3-graph $H[U]$ is maximal.
As $H[U]\supset C$, by Proposition \ref{6}, $H[U]$ is isomorphic to one of the 3-graphs $H_i(6)$, $i=1,\dots,11$.

Since $H$ is maximal, it consists of all triples  containing any 2-cover of
$H[U]$ and a vertex outside $U$. Hence, if $H[U]\cong H_i(6)$, then $H_i(n)\cong H_i(n)$, $i=1,\dots,11$. This, in view of Lemma \ref{mc}
and the `heartless' case 1, proves all parts   of Theorem \ref{main}. \qed

\section{Concluding Remarks}

Upon completing this project, we realized that the maximal intersecting 3-graphs with $\tau=3$, can
be fished out from a huge family of so called \emph{1-special} 3-graphs described in \cite{HY} (see
Theorem 5 therein). However, the authors of \cite{HY} admit that their family contains several
isomorphic 3-graphs and do not provide any proof. Also recently, we noticed that independently of
us, Kostochka and Mubayi \cite{KM} (see Theorem 8 therein) determined all maximal intersecting
3-graphs with more than 10 edges.

Although, both these results together can be, in principle, used to derive the main results of this
paper, we feel that our streamlined  and unified approach, as well as the statement in terms of the
hierarchy of Tur\'an numbers might still be interesting. Moreover, in \cite{KM} the authors attempt
to describe all maximal, intersecting $k$-graphs for $k\ge4$. Their result is, however, restricted
to $k$-graphs with large number of vertices and large number of edges. We believe that our approach
has the potential to be generalized to all $k$-graphs.


\end{document}